\documentclass{article}

\usepackage{a4wide}
\usepackage{amsmath}
\usepackage{amsthm}
\usepackage{amssymb}

\usepackage{xcolor}

\usepackage[colorlinks=true,urlcolor=blue,linkcolor=blue,plainpages=false,pdfpagelabels
]{hyperref}

\newtheorem{theorem}{Theorem}[section]
\newtheorem{proposition}[theorem]{Proposition}
\newtheorem{corollary}[theorem]{Corollary}
\newtheorem{lemma}[theorem]{Lemma}
\newtheorem{remark}[theorem]{Remark}

\newcommand{\N}{{\mathbb N}}
\newcommand{\R}{{\mathbb R}}

\newcommand{\honealpha}{(H1\alpha_n)}
\newcommand{\sigmaonealpha}{\sigma_1}
\newcommand{\htwoalpha}{(H2\alpha_n)}
\newcommand{\sigmatwoalpha}{\sigma_2}
\newcommand{\hthreealpha}{(H3\alpha_n)}
\newcommand{\sigmathreealpha}{\sigma_3}
\newcommand{\honelambda}{(H1\lambda_n)}
\newcommand{\sigmaonelambda}{\gamma_1}
\newcommand{\htwolambda}{(H2\lambda_n)}
\newcommand{\hthreelambda}{(H3\lambda_n)}
\newcommand{\sigmathreelambda}{\gamma_3}
\newcommand{\boundlambdam}{\Lambda_m}
\newcommand{\honeen}{(H1e_n)}
\newcommand{\sigmaoneen}{\theta_1}
\newcommand{\htwoen}{(H2e_n)}
\newcommand{\sigmatwoen}{\theta_2}
\newcommand{\hthreeen}{(H3e_n)}
\newcommand{\sigmathreeen}{E}

\newcommand{\KzVAM}{K_z^*}

\newcommand{\chiVAM}{\chi^*}
\newcommand{\phiVAM}{\Phi^*}
\newcommand{\psiVAM}{\Psi^*}
\newcommand{\thetaVAM}{\Theta^*}
\newcommand{\ebase}{e^*}

\newcommand{\honestarlambda}{(H1\lambda_n^*)}
\newcommand{\sigmaonestarlambda}{\gamma_1^*}

\title{Quantitative asymptotic regularity of the VAM iteration with error terms for $m$-accretive operators in Banach spaces}
\author{Paulo Firmino${}^{a}$ and Lauren{\c t}iu Leu{\c s}tean${}^{b,c,d}$\\[2mm]
\footnotesize ${}^{a}$ Departamento de Matem\'atica, Faculdade de Ci\^encias, Universidade de Lisboa\\ 
\footnotesize ${}^{b}$ LOS, Faculty of Mathematics and Computer Science, University of Bucharest\\
\footnotesize ${}^{c}$ Simion Stoilow Institute of Mathematics of the Romanian Academy\\
\footnotesize ${}^{d}$ Institute for Logic and Data Science, Bucharest\\[1mm]
\footnotesize Emails: \protect\url{fc49883@alunos.ciencias.ulisboa.pt }, \protect\url{laurentiu.leustean@unibuc.ro}
}

\begin{document}

\date{}

\maketitle

\begin{abstract}

\noindent In this paper we obtain, by using proof mining methods, quantitative results on the asymptotic regularity of the viscosity approximation method (VAM) with error  terms for $m$-accretive operators in Banach spaces. For concrete instances of the parameter sequences, linear rates are computed by applying a lemma due to Sabach and Shtern.\\



\end{abstract}

\section{Introduction}

Let $X$ be a normed space,  $A:X\to 2^X$ be an accretive operator with a nonempty set 
of zeros, and  $C\subseteq X$ be a nonempty closed convex subset of $X$ such that 
$\overline{\mathrm{dom}A} \subseteq C\subseteq  \mathrm{ran}(\mathrm{Id}+\gamma A)$ for all $\gamma>0$.  Xu et al. \cite{XuAltImtSou22} studied recently the following iteration:
\begin{equation}\label{def-VAM}
\mathrm{VAM} \qquad x_0=x, \quad x_{n+1}=\alpha_n f(x_n) +(1-\alpha_n)J_{\lambda_n}^Ax_n,
\end{equation}
where $x\in C$, $f:C\to C$ is an $\alpha$-contraction for $\alpha\in[0,1)$, $(\lambda_n)$ is 
a sequence in $(0,\infty)$, $(\alpha_n)$ is a sequence in $[0,1]$, and, for every $n\in\N$, $J_{\lambda_n}^A$ is 
the resolvent of order $\lambda_n$ of $A$.

The VAM  iteration is an instance of the viscosity approximation method 
applied to resolvents of accretive operators in Banach spaces (see, for example, \cite{KopRei08,Rei80,Rei94,Att96,Mou00,Xu04}). If one takes $f(x)=u\in X$ in 
\eqref{def-VAM}, one gets the Halpern-type Proximal point algorithm HPPA, introduced by  
Kamimura and Takahashi \cite{KamTak00} and Xu \cite{Xu02}, a modification of the Proximal point 
algorithm that was studied in a series of papers in recent years.  Thus, VAM is a viscosity 
version of the HPPA.

Xu et al. \cite{XuAltImtSou22} proved, in the setting of uniformly convex 
and/or uniformly G\^ateaux differentiable Banach spaces,  
strong convergence results for the VAM iteration towards a zero of $A$, extending results for the HPPA obtained by 
Aoyama and Toyoda \cite{AoyToy17}. 

As it is the case with numerous convergence proofs, an intermediate step is to obtain the asymptotic regularity of the iteration.
Asymptotic regularity  was introduced by Browder and Petryshyn \cite{BroPet66}  for the Picard iteration and extended to general iterations by Borwein, Reich, and Shafrir \cite{BorReiSha92}. 
By inspecting the proofs from  \cite{XuAltImtSou22}, one can see that 
asymptotic regularity of the VAM iteration holds, under some hypotheses on $(\alpha_n)$, $(\lambda_n)$,  
in the more general setting of Banach spaces. 

In this paper we prove quantitative asymptotic regularity  results for the VAMe iteration for $m$-accretive operators,  defined by adding error terms  to the VAM iteration (see \eqref{def-VAMe}). These quantitative results provide uniform rates of 
asymptotic regularity, $\left(J_{\lambda_n}^A\right)$-asymptotic regularity and, for all 
$m\in\N$,  $J_{\lambda_m}^A$-asymptotic regularity for VAMe. 
We compute linear such rates for concrete instances of the parameter sequences $(\alpha_n)$, 
$(\lambda_n)$, as an application of a lemma of Sabach and Shtern \cite{SabSht17}.  As VAM and HPPA  for $m$-accretive operators are particular cases of our VAMe iteration, we obtain rates for these iterations, too.
Furthermore, as an immediate consequence of our quantitative results, we obtain 
qualitative  asymptotic regularity results for the VAMe iteration.

The results from the paper are obtained by applying methods of proof mining, a research program 
concerned with the extraction, by using proof-theoretic techniques, of new quantitative and 
qualitative information from mathematical proofs.  We refer to Kohlenbach's textbook \cite{Koh08} 
for details on proof mining and to \cite{Koh19,Koh20} for surveys of recent applications in nonlinear analysis and optimization. 
Finally, let us remark that proof mining was applied recently by Kohlenbach and Pinto \cite{KohPin22} to obtain quantitative results, providing rates of metastability, for viscosity approximation methods in $W$-hyperbolic spaces.

\section{Preliminaries}

Let $X$ be a normed space and $A:X \to 2^X$ be a set-valued operator on $X$. As usual, we identify the operator $A$ with its graph 
$ \mathrm{gra} A=\{(x,y)\in X \times X \mid  y\in Ax\}$. Let $\mathrm{dom}A=\{x\in X \mid  Ax \neq \emptyset\}$ be the domain of $A$ and 
$\mathrm{ran}A=\bigcup_{x\in X} Ax$ be the range of $A$. Furthermore, we denote by $\mathrm{zer} A$ the set of zeros of $A$, that is
$\mathrm{zer} A=\{x\in X\mid 0\in Ax\}$. The definition of the inverse $A^{-1}$ of $A$ is given through its graph: 
$\mathrm{gra} A^{-1}=\{(y,x)\in X \times X \mid (x,y)\in \mathrm{gra} A\}= \{(y,x)\in X \times X \mid y\in Ax\}$. If $\lambda \in\R$ and $B$ is another set-valued operator on $X$, then
$\lambda A =\{(x,\lambda y)\mid x\in X, y\in Ax\}$ and $A+B=\{(x, y+z)\mid x\in X, y\in Ax, z\in Bx\}$. 
For every $\gamma>0$, the resolvent $J_{\gamma}^A$ of order $\gamma$ of $A$ is  defined by $J_{\gamma}^A=(\mathrm{Id}+\gamma A)^{-1}$, where $\mathrm{Id}$ 
is the identity operator on $X$. One can easily verify that $\mathrm{dom}J_{\gamma}^A=\mathrm{ran}(\mathrm{Id}+\gamma A)$ and  $\mathrm{ran}J_{\gamma}^A = \mathrm{dom}A$. 

Let us recall that if $\emptyset \ne C\subseteq X$ and  $T:C\to X$ is a mapping, we denote by $\mathrm{Fix}(T)$ 
the set of fixed points of $T$ and $T$ is said to be nonexpansive if $\|Tx-Ty\|\le \|x-y\|$ for all $x,y\in X$.

An operator $A$ is said to be \emph{accretive} \cite{Bro67,Kat67} if for all $x,y\in \mathrm{dom}A$, $u\in Ax$, $v\in Ay$, and $\gamma>0$,
$$
\|x-y+\gamma(u-v)\|\geq \|x-y\|.
$$

It is well-known that for any accretive operator $A$ and for all $\gamma>0$, 
$J_{\gamma}^A:\mathrm{ran}(\mathrm{Id}+\gamma A)\to \mathrm{dom}A$  is a nonexpansive mapping such that $\mathrm{Fix}(J_{\gamma}^A)=\mathrm{zer} A$ 
(see, for example, \cite[Corollary 3.4.1]{GarKha23}  and \cite[Proposition 6.7.1]{GarKha23} for proofs). 

\begin{lemma}
Assume that $A$ is an accretive operator. 
Let  $\lambda,\gamma>0$.
\begin{enumerate}
\item If  $x\in \mathrm{ran}(\mathrm{Id}+\lambda A)$, then $\displaystyle \frac{\gamma}{\lambda}x+\left(1-\frac{\gamma}{\lambda}\right)J^A_\lambda x \in \mathrm{ran}(\mathrm{Id}+\gamma A)$ and 
\begin{align} 
J^A_\lambda x & =J^A_\gamma\left(\frac{\gamma}{\lambda}x+\left(1-\frac{\gamma}{\lambda}\right)J^A_\lambda x\right). \label{J-eq}
\end{align}
\item For all $x\in \mathrm{ran}(\mathrm{Id}+\lambda A)\cap \mathrm{ran}(\mathrm{Id}+\gamma A)$, 
\begin{align}
\|J^A_{\gamma}x-J^A_{\lambda}x\| & \le \left|1-\frac{\gamma}{\lambda}\right| \left\|J^A_{\lambda} x-x \right\|. \label{J-ineq}
\end{align}
\end{enumerate}
\end{lemma}
\begin{proof}
For a proof of \eqref{J-eq} see \cite[Proposition 3.4.1]{GarKha23}. \eqref{J-ineq} follows immediately from \eqref{J-eq} and the fact that $J^A_{\gamma}$ is nonexpansive:
\begin{align*}
\|J^A_{\gamma}x-J^A_{\lambda}x\|  & = 
\left\|J^A_{\gamma}x- J^A_{\gamma}\left(\frac{\gamma}{\lambda}x+
\left(1-\frac{\gamma}{\lambda}\right)J^A_{\lambda} x\right)\right\| \\
&  \leq \left\|x- \frac{\gamma}{\lambda}x-
\left(1-\frac{\gamma}{\lambda}\right)J^A_{\lambda} x   \right\|= \left|1-\frac{\gamma}{\lambda}\right| \left\|x - J^A_{\lambda} x \right\|.
\end{align*}
\end{proof}

An  \emph{$m$-accretive}  operator is an accretive operator $A$  that satisfies $\mathrm{ran}(\mathrm{Id}+\gamma A)=X$  for all $\gamma >0$. 
It follows that, for an $m$-accretive operator $A$, \eqref{J-eq}  and \eqref{J-ineq} hold for all $x\in X$.

\section{Quantitative notions and lemmas}

Let us recall the main quantitative notions that will be used in this paper. 
Suppose that $(a_n)_{n\in\N}$ is a sequence in a metric space $(X,d)$. A mapping  $\varphi:\N\to\N$ 
is said to be 
\begin{enumerate}
\item a Cauchy modulus of $(a_n)$ if for all $k\in\N$ and all $n\geq \varphi(k)$,
$$d(a_{n+p},a_n)\leq 
\frac1{k+1} \quad \text{holds for all~}p\in\N.$$
\item  a rate of convergence of $(a_n)$ (towards $a\in X$) if for all $k\in\N$ and all $n\geq \varphi(k)$,
\[d(a_n,a)\leq \frac1{k+1}.\]
\end{enumerate}

Obviously, $(a_n)$ is Cauchy iff $(a_n)$ has a Cauchy modulus, and $\lim\limits_{n\to\infty} a_n=a$ iff  
$(a_n)$ has a rate of convergence towards $a$.

Assume that $\sum\limits_{n=0}^\infty b_n$ is a series of nonnegative real numbers and 
$\left(\tilde{b}_n=\sum\limits_{i=0}^{n} b_i\right)$ is the sequence of partial sums. 
Then a Cauchy modulus of the series is  a Cauchy modulus of $\left(\tilde{b}_n\right)$. 
A rate of divergence of the series is a mapping $\theta:\N\to\N$  satisfying 
$\sum\limits_{i=0}^{\theta(n)} b_i \geq n$ for all $n\in\N$. It is clear that 
$\sum\limits_{n=0}^\infty b_n$ diverges iff it has a rate of divergence. \\

Let  $(y_n)$ be a sequence in a metric space $(X,d)$, $\emptyset \neq C\subseteq X$, $T:C\to C$, $(T_n : C \to C)_{n\in\N}$ be 
a countable family of mappings, and $\Phi:\N\to\N$. We say that 
\begin{enumerate}
\item $(y_n)$ is asymptotically regular with rate $\Phi$ \big(or $\Phi$  is a rate of asymptotic regularity of $(y_n)$\big) 
if $\lim\limits_{n\to \infty} d(y_n,y_{n+1})=0$ with rate of convergence $\Phi$;
\item $(y_n)$ is $T$-asymptotically regular with rate $\Phi$ \big(or  $\Phi$  is a rate of $T$-asymptotic regularity of $(y_n)$\big)
if $\lim\limits_{n\to \infty} d(y_n,Ty_n)=0$ with rate of convergence $\Phi$;
\item $(y_n)$ is $(T_n)$-asymptotically regular with rate $\Phi$ \big(or  $\Phi$  is a rate of $(T_n)$-asymptotic regularity of $(y_n)$\big) 
if $\lim\limits_{n\to \infty} d(y_n,T_ny_n)=0$ with rate of convergence $\Phi$.
\end{enumerate} 

\subsection{Useful lemmas on sequences of real numbers}

\begin{lemma}\label{thetngeqn-2}
If $(b_n)$ is a sequence in $[0,1]$ and $\theta$ is a rate of divergence for $\sum\limits_{n=0}^\infty b_n$, then
$\theta(n)\geq n-2$ for all $n\in\N$.
\end{lemma}
\begin{proof}
Assume  that $\theta(n) <  n-2$ for some  $n\in\N$. It follows that 
$
\sum\limits_{i=0}^{\theta(n)} b_i  \leq \sum\limits_{i=0}^{n-2} b_i  \leq n-1<n,
$
which is a contradiction.
\end{proof}

\begin{lemma}\label{cauchy-modulus-linear-comb}
Let $(a_n)$, $(b_n)$ be sequences of nonnegative real numbers, $p,q\in\N$, and $c_n=pa_n+qb_n$ for all $n\in\N$. 
Assume that $(a_n)$ is Cauchy with Cauchy modulus $\varphi_1$ and 
$(b_n)$ is Cauchy with Cauchy modulus $\varphi_2$. Then $(c_n)$ is Cauchy with Cauchy modulus 
$$ \varphi(k)=\max\{\varphi_1(2p(k+1)-1), \varphi_2(2q(k+1)-1)\}.$$
\end{lemma}
\begin{proof}
Let $k\in\N$.  We get that for all $n\geq \varphi(k)$ and all $p\in \N$,  
\begin{align*}
c_{n+p} - c_n &=  p(a_{n+p}-a_n) + q (b_{n+p}-b_n) \leq p\cdot\frac1{2p(k+1)} + q\cdot\frac1{2q(k+1)}=\frac1{k+1}.
\end{align*} 
\end{proof}

The following result is \cite[Proposition 2.7]{LeuPin24}, which is a reformulation of \cite[Lemma 2.9(1)]{DinPin23}, 
obtained by taking $\displaystyle \frac1{k+1}$ instead of $\varepsilon$. 
It is a quantitative version of a particular case of a very useful lemma on 
sequences of real numbers due to Xu \cite{Xu02}.

\begin{proposition}\label{quant-lem-Xu02-bn-0}
Let $(a_n)$ be a sequence in $[0,1]$ and 
$(c_n), (s_n)$ sequences of nonnegative real numbers  such that for all $n\in\N$,  
\begin{equation}\label{def-sn-an-cn}
s_{n+1}\leq (1-a_n)s_n + c_n. 
\end{equation}
Assume that $L\in\N^*$ is an upper bound on $(s_n)$, $\sum\limits_{n=0}^\infty a_n$  diverges with rate of 
divergence $\theta$, and $\sum\limits_{n=0}^\infty c_n$ converges with Cauchy modulus $\chi$.

Then $\lim\limits_{n\to\infty} s_n=0$ with rate of convergence $\Sigma$ defined by 
\begin{align*}
\Sigma(k) &=\theta\left(\chi(2k+1)+1+\lceil \ln(2L(k+1))\rceil\right)+1.
\end{align*} 
\end{proposition}

The next lemma is a slight variation of \cite[Lemma~3]{SabSht17}, proved in \cite[Lemma~2.8]{LeuPin24}. 

\begin{lemma}\label{lem:sabach-shtern-v2}
Let $L > 0$, $J \geq N \geq 2$, $\gamma \in (0, 1]$, $(c_n)$ be a sequence bounded above by $L$, and 
$a_n = \frac{N}{\gamma(n + J)}$ for all $n\in\N$. Suppose that $(s_n)$ is a sequence of nonnegative 
real numbers such that $s_0 \leq L$ and, for all $n \in \N$,
\begin{equation*}
s_{n + 1} \leq (1 - \gamma a_{n + 1}) s_n + (a_n - a_{n + 1}) c_n. 
\end{equation*}
Then, for all $n \in \N$,
\begin{equation*}
s_n \leq \frac{J L}{\gamma(n + J)}.
\end{equation*}
\end{lemma}

\section{VAM with errors for resolvents of $m$-accretive operators in Banach spaces}\label{VAMe-results}

Let $X$ be a normed space, $A:X\to 2^X$ be an $m$-accretive operator such that $\mathrm{zer} A\ne\emptyset$,  
and $f:X\to X$ be an $\alpha$-contraction for $\alpha\in[0,1)$, 
that is $\|f(x)-f(y)\|\leq \alpha \|x-y\|$ for all $x,y\in X$. 

We consider the iteration $(x_n)$ defined as follows: 
\begin{equation}\label{def-VAMe}
\mathrm{VAMe} \qquad x_0=x\in X, \quad x_{n+1}=\alpha_n f(x_n) +(1-\alpha_n)J_{\lambda_n}^Ax_n + e_n,
\end{equation}
where $(\alpha_n)_{n\in\N}$ is a sequence in $[0,1]$,  $(\lambda_n)_{n\in\N}$ is a sequence in $(0,\infty)$, and $(e_n)_{n\in\N}$ is 
a sequence in $X$. Hence, $(x_n)$ is obtained from the VAM iteration studied in \cite{XuAltImtSou22} 
by adding error terms $e_n$. 

For every $z\in \mathrm{zer}A$, let $(K_{z,n})_{n\in\N}$ be a sequence of real numbers defined as follows:
\begin{equation}\label{def-Kzn}
K_{z,0}=\max\left\{\|x-z\|,\frac{\|f(z)-z\|}{1-\alpha}\right\}, \quad K_{z,n}=K_{z,0}+\sum_{i=0}^{n-1}\|e_i\| \text{~~for all~}n\geq 1.
\end{equation}
Thus, $K_{z,n+1}=K_{z,n}+\|e_n\|$ for all $n\geq 0$. 

\begin{lemma}\label{xn-bound-as-reg}
For all $z\in \mathrm{zer}A$ and $m,n\in\N$, 
\begin{enumerate}
\item\label{xnz-fxnz-bound-Kzn} $\|x_n-z\|,\|f(x_n)-z\|\le K_{z,n}$; 
\item\label{xn+xn-bound-2Kzn+1} $\|x_{n+1}-x_n\|\leq 2K_{z,n+1}$;
\item\label{Jlambdamxnz-bound-Kzn} $\left\|J^A_{\lambda_m}x_n-z \right\|\le K_{z,n}$;
\item\label{Jlambdamxnxnfxnz-bound-2Kzn} $\left\|J^A_{\lambda_m}x_n-x_n\right\|, \left\|J^A_{\lambda_m}x_n-f(x_n)\right\| \leq 2K_{z,n}$.
\end{enumerate}
\end{lemma}
\begin{proof}
\begin{enumerate}
\item We prove the two inequalities simultaneously by induction on $n$.

$n=0$: $\|x_0-z\|\leq  K_{z,0}$ follows by \eqref{def-Kzn}. Furthermore, applying the fact that 
$f$ is an $\alpha$-contraction and \eqref{def-Kzn}, we get that 
$$\|f(x_0)-z\|  \le \|f(x_0)-f(z)\|+\|f(z)-z\| \le \alpha \|x_0-z\|+ (1-\alpha)K_{z,0} \le K_{z,0}.$$

$n\Rightarrow n+1$: We have that 
\begin{align*}
\|x_{n+1}-z\|&=\left\|\alpha_n (f(x_n)-z) +(1-\alpha_n)(J_{\lambda_n}^Ax_n-J_{\lambda_n}^Az)+e_n \right\|
 \quad \text{~~as~} J_{\lambda_n}^Az=z\\
&\le \alpha_n \|f(x_n)-z\|+(1-\alpha_n)\|x_n-z\|+\|e_n\| \quad \text{~~as~} J_{\lambda_n}^A \text{~is nonexpansive}\\
&\le K_{z,n}+\|e_n\| \quad \text{~~by the induction hypothesis}\\
& =  K_{z,n+1}.
\end{align*}
Moreover, $\|f(x_{n+1})-z\|\leq \alpha\|x_{n+1}-z\|+\|f(z)-z\| \le K_{z,n+1}$.
\item $\|x_{n+1}-x_n\| \le \|x_n-z\|+\|x_{n+1}-z\| \le K_{z,n} + K_{z,n+1} \le 2K_{z,n+1}$.  
\item $\left\|J^A_{\lambda_m}x_n-z\right\| = \left\|J^A_{\lambda_m}x_n-J^A_{\lambda_m}z\right\| \le \|x_n-z\| \le K_{z,n}$.
\item $\left\|J^A_{\lambda_m}x_n-x_n\right\| \le \left\|J^A_{\lambda_m}x_n-z\right\| + \|x_n-z\| \le 2K_{z,n}$ and 
$\left\|J^A_{\lambda_m}x_n-f(x_n)\right\|  \le \left\|J^A_{\lambda_m}x_n-z\right\| + \|f(x_n)-z\| \le 2K_{z,n}$.
\end{enumerate}
\end{proof}

The following is the main inequality that will be used in the proof of one of our main results from 
Section~\ref{section-main-results}. 

\begin{proposition}
For all $n\in\N$,
\begin{align}
\|x_{n+2}-x_{n+1}\| & \le  (1-(1-\alpha)\alpha_{n+1})\|x_{n+1}-x_n\| + M_{z,n} + \left\|e_{n+1}-e_n\right\|, \label{main-ineq-VAMe-lambdaone}\\
\|x_{n+2}-x_{n+1}\| & \le  (1-(1-\alpha)\alpha_{n+1})\|x_{n+1}-x_n\| + M^*_{z,n} + \left\|e_{n+1}-e_n\right\|, \label{main-ineq-VAMe-lambdastarone}
\end{align}
where 
\begin{align*}
M_{z,n} &= 2K_{z,n}\left(|\alpha_{n+1}-\alpha_n|
+ (1-\alpha_{n+1})\left| 1-\frac{\lambda_{n+1}}{\lambda_n} \right|\right),\\
M^*_{z,n} &= 2K_{z,n}\left(|\alpha_{n+1}-\alpha_n|
+ (1-\alpha_{n+1})\left| 1-\frac{\lambda_n}{\lambda_{n+1}} \right|\right).
\end{align*}
\end{proposition}
\begin{proof} We have  that 
\begin{align*}
x_{n+2}-x_{n+1} &= \left(\alpha_{n+1}f(x_{n+1})+(1-\alpha_{n+1})J^A_{\lambda_{n+1}}x_{n+1}\right)-
\left(\alpha_{n+1}f(x_n)+(1-\alpha_{n+1})J^A_{\lambda_n}x_n\right)\\
&  \quad + \left(\alpha_{n+1}f(x_n)+(1-\alpha_{n+1})J^A_{\lambda_n}x_n\right) - 
\left(\alpha_nf(x_n)+(1-\alpha_n)J^A_{\lambda_n}x_n\right)+e_{n+1}-e_n\\
&= \alpha_{n+1}\left(f(x_{n+1})-f(x_n)\right) + (1-\alpha_{n+1})\left(J^A_{\lambda_{n+1}}x_{n+1}-J^A_{\lambda_n}x_n\right)\\
& \quad  + \left(\alpha_{n+1}-\alpha_n\right)f(x_n) + (\alpha_n-\alpha_{n+1})J^A_{\lambda_n}x_n+e_{n+1}-e_n\\
&= \alpha_{n+1}\left(f(x_{n+1})-f(x_n)\right) + (1-\alpha_{n+1})\left(J^A_{\lambda_{n+1}}x_{n+1}-J^A_{\lambda_n}x_n\right)\\
& \quad  + \left(\alpha_{n+1}-\alpha_n\right)\left(f(x_n)-J^A_{\lambda_n}x_n\right)+e_{n+1}-e_n.
\end{align*}
Thus, 
\begin{align*}
\|x_{n+2}-x_{n+1}\| &\le  \alpha_{n+1}\alpha\|x_{n+1}-x_n\| + (1-\alpha_{n+1})\left\|J^A_{\lambda_{n+1}}x_{n+1}-J^A_{\lambda_n}x_n\right\| \\
&  \quad  + \left|\alpha_{n+1}-\alpha_n\right|\left\|f(x_n)-J^A_{\lambda_n}x_n\right\| + \|e_{n+1}-e_n\|\\
 &\le  \alpha_{n+1}\alpha\|x_{n+1}-x_n\| + (1-\alpha_{n+1})\left\|J^A_{\lambda_{n+1}}x_{n+1}-J^A_{\lambda_n}x_n\right\|\\
 & \quad  + 2K_{z,n}\left|\alpha_{n+1}-\alpha_n\right| + \|e_{n+1}-e_n\|  
 \quad  \text{by Lemma~\ref{xn-bound-as-reg}.\eqref{Jlambdamxnxnfxnz-bound-2Kzn}}.
\end{align*}
As
\begin{align*}
\left\|J^A_{\lambda_{n+1}}x_{n+1}-J^A_{\lambda_n}x_n\right\| & \le 
\left\|J^A_{\lambda_{n+1}}x_{n+1}-J^A_{\lambda_{n+1}}x_n\right\| + \left\|J^A_{\lambda_{n+1}}x_n-J^A_{\lambda_n}x_n \right\|\\
& \le  \left\|x_{n+1}-x_n\right\| + \left\|J^A_{\lambda_{n+1}}x_n-J^A_{\lambda_n}x_n \right\|,
\end{align*}
it follows that 
\begin{align*}
\|x_{n+2}-x_{n+1}\| &\le  \left(\alpha_{n+1}\alpha+1-\alpha_{n+1}\right)\|x_{n+1}-x_n\| + 
(1-\alpha_{n+1})\left\|J^A_{\lambda_{n+1}}x_n-J^A_{\lambda_n}x_n \right\| \\
&  \quad + 2K_{z,n}\left|\alpha_{n+1}-\alpha_n\right| + \|e_{n+1}-e_n\|.
\end{align*}

By \eqref{J-ineq} and Lemma~\ref{xn-bound-as-reg}.\eqref{Jlambdamxnxnfxnz-bound-2Kzn} we have that 
\begin{align}
\left\|J^A_{\lambda_{n+1}}x_n-J^A_{\lambda_n}x_n\right\| &\le   \left|1-\frac{\lambda_{n+1}}{\lambda_n}\right| \left\|J^A_{\lambda_n} x_n-x_n \right\| 
\le 2K_{z,n}\left|1-\frac{\lambda_{n+1}}{\lambda_n}\right|, \label{useful-ineq-Mzn}\\
\left\|J^A_{\lambda_{n+1}}x_n-J^A_{\lambda_n}x_n\right\| &\le   \left|1-\frac{\lambda_n}{\lambda_{n+1}}\right| \left\|J^A_{\lambda_{n+1}} x_n-x_n \right\| 
\le 2K_{z,n}\left|1-\frac{\lambda_n}{\lambda_{n+1}}\right|. \label{useful-ineq-Mstarzn}
\end{align}
Apply \eqref{useful-ineq-Mzn} and  \eqref{useful-ineq-Mstarzn} to conclude that \eqref{main-ineq-VAMe-lambdaone} and \eqref{main-ineq-VAMe-lambdastarone}
hold.
\end{proof}

\subsection{Quantitative hypotheses on the parameter sequences}

We consider the following hypotheses on the parameter sequences $(\alpha_n)$, $(\lambda_n)$, $(e_n)$ 
from the definition \eqref{def-VAMe} of the VAMe iteration $(x_n)$:
\begin{align*}
\honealpha & \quad \sum\limits_{n=0}^{\infty} \alpha_n =\infty \text{~with divergence rate~} \sigmaonealpha;\\
\htwoalpha & \quad \sum\limits_{n=0}^{\infty} |\alpha_n-\alpha_{n+1}| <\infty \text{~with Cauchy modulus~} \sigmatwoalpha;\\
\hthreealpha & \quad \lim\limits_{n\to\infty}\alpha_n=0 \text{~with rate of convergence~} \sigmathreealpha;\\
\honelambda & \quad \sum\limits_{n=0}^{\infty} \left|1-\frac{\lambda_{n+1}}{\lambda_n}\right|<\infty 
\text{~with Cauchy modulus~} \sigmaonelambda;\\
\honestarlambda & \quad \sum\limits_{n=0}^{\infty} \left|1-\frac{\lambda_n}{\lambda_{n+1}}\right|<\infty \text{~with Cauchy modulus~} 
\sigmaonestarlambda;\\
\htwolambda & \quad \Lambda\in\N^* \text{~and~} N_{\Lambda}\in\N \text{~are such that~} \lambda_n \geq 
\frac1\Lambda \text{~for all~} n\geq N_{\Lambda}; \\
\hthreelambda & \quad \sum\limits_{n=0}^{\infty} |\lambda_n-\lambda_{n+1}| <\infty \text{~with Cauchy modulus~} \sigmathreelambda; \\
\honeen & \quad \sum\limits_{n=0}^{\infty} \|e_n\| <\infty \text{~with Cauchy modulus~} \sigmaoneen; \\
\htwoen & \quad \lim\limits_{n\to\infty}\|e_n\|=0 \text{~with rate of convergence~} \sigmatwoen; \\
\hthreeen & \quad  \sigmathreeen\in \N^* \text{~is an upper bound on~} \sum\limits_{n=0}^{\infty} \|e_n\|.
\end{align*}

\begin{lemma}\label{honeen-htwoen-htwoen-Kz}$\,$
\begin{enumerate}
\item\label{hthreeen-Kz} Assume that $\hthreeen$ holds. For every $z\in \mathrm{zer}A$, let $K_z\in \N^*$ be such that 
\begin{equation*}
K_z\geq \max\left\{\|x-z\|,\frac{\|f(z)-z\|}{1-\alpha}\right\}+\sigmathreeen.
\end{equation*}
Then $K_{z,n}\leq K_z$ for all $n\in \N$. 
Hence, Lemma~\ref{xn-bound-as-reg} and inequalities \eqref{main-ineq-VAMe-lambdaone}, \eqref{main-ineq-VAMe-lambdastarone} 
hold with $K_z$ instead of $K_{z,n}$ or $K_{z,n+1}$. 
\item\label{honeen-htwoen-hthreen} Suppose that $\honeen$ holds. Then $\htwoen$ is satisfied with $\sigmatwoen(k)=\sigmaoneen(k)+1$ and  
$\hthreeen$ is satisfied with $\sigmathreeen=\left\lceil\sum\limits_{i=0}^{\sigmaoneen(0)}\|e_i\|\right\rceil+1$. 
\end{enumerate}
\end{lemma}
\begin{proof}
Let us denote, for all $m\in\N$, $\tilde{e}_m=\sum\limits_{i=0}^{m}\|e_i\|$.
\begin{enumerate}
\item Obviously, by \eqref{def-Kzn}.
\item Let $k\in\N$ and $n\geq \sigmatwoen(k)$
We get that $\|e_n\| = \tilde{e}_n - \tilde{e}_{n-1} \leq \frac1{k+1}$, as  $n-1\geq \sigmaoneen(k)$, 
\text{~so~ we can apply~}  $\honeen$.

Obviously, if $n<\sigmaoneen(0)$, we have that $\tilde{e}_n \le \tilde{e}_{\sigmaoneen(0)} <\sigmathreeen$. 
Let $n\geq \sigmaoneen(0)$. By $\honeen$, we get that $\tilde{e}_n - \tilde{e}_{\sigmaoneen(0)}\leq 1$,
hence $\tilde{e}_n \leq \sigmathreeen$. 
\end{enumerate}
\end{proof}

\begin{lemma}\label{htwolambda+hthreelambda-implies-honelambda}
Assume $\htwolambda$ and $\hthreelambda$. Then $\honelambda$ and $\honestarlambda$ hold with 
\begin{align}\label{def-sigmaonelambda}
\sigmaonelambda(k)  = \sigmaonestarlambda(k) & = \max\{N_\Lambda,\sigmathreelambda(\Lambda (k+1)-1)\}.
\end{align}
\end{lemma}
\begin{proof}
Let us denote $\tilde{\lambda}_n=\sum\limits_{i=0}^n\left|\lambda_i-\lambda_{i+1}\right|$. We get that for all $n\geq \sigmaonelambda(k)$ and all $p\in\N$,
\begin{align*}
\sum_{i=0}^{n+p}\left|1-\frac{\lambda_{i+1}}{\lambda_i}\right| -
\sum_{i=0}^n\left|1-\frac{\lambda_{i+1}}{\lambda_i}\right| & =  \sum_{i=n+1}^{n+p}\frac{1}{\lambda_i}\left|\lambda_i-\lambda_{i+1}\right|
\leq  \sum_{i=n+1}^{n+p}\Lambda\left|\lambda_i-\lambda_{i+1}\right| \quad \text{by~} \htwolambda \\
& = \Lambda \left(\tilde{\lambda}_{n+p} - \tilde{\lambda}_n\right)  \leq   \frac1{k+1} \quad\text{by~} \hthreelambda.
\end{align*}
The fact that $\honestarlambda$ holds is obtained similarly. 
\end{proof}

\section{Rates of asymptotic regularity, $\left(J_{\lambda_n}^A\right)$-asymptotic regularity and, for all 
$m\in\N$,  $J_{\lambda_m}^A$-asymptotic regularity}
\label{section-main-results}

Throughout this section, $X$ is a Banach space, $A:X\to 2^X$ is an $m$-accretive operator such that $\mathrm{zer} A\ne\emptyset$, 
$f:X\to X$ is an $\alpha$-contraction for $\alpha\in[0,1)$, $x\in X$, and $(x_n)$ is the VAMe iteration starting with $x$,
defined by \eqref{def-VAMe}.

The first main result of the paper gives effective rates of asymptotic regularity of $(x_n)$.

\begin{theorem}\label{main-rate-as-reg-VAMe-1}
Suppose that $\honealpha$, $\htwoalpha$, $\honelambda$, and $\honeen$ hold. 
Let $z\in \mathrm{zer}A$, $K_z\in \N^*$ be such that 
\begin{align}\label{def-Kz-main}
K_z & \geq \max\left\{\|x-z\|,\frac{\|f(z)-z\|}{1-\alpha}\right\} + \left\lceil\sum\limits_{i=0}^{\sigmaoneen(0)}\|e_i\|\right\rceil + 1,
\end{align}
and 
\begin{align*}
\chi(k) & =\max\{\sigmatwoalpha(6K_z(k+1)-1),\sigmaonelambda(6K_z(k+1)-1),\sigmaoneen(6k+5)\}.
\end{align*}
Then $(x_n)$ is asymptotically regular with rate $\Phi:\N\to \N$ defined by 
\begin{align*}
\Phi(k) & = \sigmaonealpha\left(\left \lceil \frac{\chi(2k+1)+1+\lceil \ln(4K_z(k+1))\rceil}{1-\alpha} \right\rceil 
+1\right).
\end{align*}
\end{theorem}
\begin{proof}
We show that we can apply Proposition~\ref{quant-lem-Xu02-bn-0} with $ s_n = \|x_{n+1}-x_n\|$, $L=2K_z$, 
$$ \, \, a_n=(1-\alpha)\alpha_{n+1},\, \, \text{and} \,\,
c_n = 2K_z\left(|\alpha_{n+1}-\alpha_n|+\left|1-\frac{\lambda_{n+1}}{\lambda_n}\right|\right)+\|e_{n+1}-e_n\|.
$$ 

Let us remark first that \eqref{def-sn-an-cn} holds, as a consequence of \eqref{main-ineq-VAMe-lambdaone}
and Lemma~\ref{honeen-htwoen-htwoen-Kz}.\eqref{hthreeen-Kz}. 
Furthermore, by Lemmas~\ref{xn-bound-as-reg}.\eqref{xn+xn-bound-2Kzn+1} and \ref{honeen-htwoen-htwoen-Kz}.\eqref{hthreeen-Kz}, 
we have that $L$ is an upper bound on $(s_n)$.

For the rest of the proof let $k\in\N$ be arbitrary. 
Define
$$ \theta(k)=\max\left\{\sigmaonealpha \left (\left \lceil \frac{k}{1-\alpha} \right\rceil +1\right )-1,0\right\}.$$
It follows that
\begin{eqnarray*}
\sum\limits_{n=0}^{\theta(k)} a_n & = & (1-\alpha)
\left(\sum\limits_{n=0}^{\theta(k)+1}\alpha_n - \alpha_0\right) \geq (1-\alpha)
\left(\sum\limits_{n=0}^{\sigmaonealpha\left (\left \lceil \frac{k}{1-\alpha} \right\rceil +1\right)}\alpha_n-\alpha_0\right)\\
& \stackrel{\honealpha}{\geq} & (1-\alpha)\left(\left\lceil \frac{k}{1-\alpha} \right\rceil +1 -\alpha_0\right) 
\geq  (1-\alpha)\left\lceil \frac{k}{1-\alpha}\right\rceil \quad \text{as~}\alpha_0\leq 1\\
& \geq & k.
\end{eqnarray*}
Thus, $\theta$ is a rate of divergence of $\sum\limits_{n=0}^\infty a_n$.

Denote, for all $m\in\N$, 
\begin{center} $\tilde{\alpha}_m=\sum\limits_{i=0}^{m}|\alpha_{i+1}-\alpha_i|$, 
$\tilde{\lambda}_m=\sum\limits_{i=0}^{m}\left|1-\frac{\lambda_{i+1}}{\lambda_i}\right|$, 
$\tilde{e}_m=\sum\limits_{i=0}^{m}\|e_i\|$, and $\tilde{c}_m=\sum\limits_{i=0}^{m}c_i$. 
\end{center}
We get that for all $n\geq \chi(k)$ and all $p\in\N^*$,
\begin{align*}
\tilde{c}_{n+p} - \tilde{c}_n & = 2K_z\left(\left(\tilde{\alpha}_{n+p} - \tilde{\alpha}_n\right) + 
\left(\tilde{\lambda}_{n+p} - \tilde{\lambda}_n\right)\right) +  \sum_{i=n+1}^{n+p} \|e_{i+1}-e_i\| \\
& \leq \frac{4K_z}{6K_z(k+1)} + \sum_{i=n+1}^{n+p} \|e_{i+1}-e_i\|\quad \text{by~} \htwoalpha \text{~and~} \honelambda\\
& \leq \frac2{3(k+1)} + \sum_{i=n+1}^{n+p} (\|e_{i+1}\|+\|e_i\|) \\
&  = \frac2{3(k+1)} + \left(\tilde{e}_{n+1+p} - \tilde{e}_{n+1}\right) + \left(\tilde{e}_{n+p} - \tilde{e}_n\right)\\
&  \leq \frac2{3(k+1)} + \frac2{6(k+1)} \quad  \text{as~} n \geq \sigmaoneen(6k+5), \text{~so we can apply~} \honeen \text{~twice}\\
& = \frac1{k+1}.
\end{align*} 
Thus, $\sum\limits_{n=0}^\infty c_n$ converges with Cauchy modulus $\chi$.

We can apply Proposition~\ref{quant-lem-Xu02-bn-0} to conclude that $\lim\limits_{n\to\infty} \|x_{n+1}-x_n\|=0$ with rate of convergence
\begin{align*}
\Sigma(k)&= \theta(P)+1 = \max\left\{\sigmaonealpha\left (\left \lceil \frac{P}{1-\alpha} \right\rceil +1\right)-1,0\right\}+1
=\max\left\{\sigmaonealpha\left (\left \lceil \frac{P}{1-\alpha} \right\rceil +1\right),1\right\},
\end{align*}
where $P=\chi(2k+1)+1+\lceil \ln(4K_z(k+1))\rceil$. As $\left \lceil \frac{P}{1-\alpha} \right\rceil +1\geq P+1 \geq 2+\lceil \ln 4\rceil=4$,
it follows, by Lemma~\ref{thetngeqn-2}, that $\sigmaonealpha\left (\left \lceil \frac{P}{1-\alpha} \right\rceil +1\right)\geq 2$, hence
$$ \Sigma(k) = \sigmaonealpha\left (\left \lceil \frac{P}{1-\alpha} \right\rceil +1\right) = \Phi(k).$$
\end{proof}

\begin{remark}\label{main-theorem-remark-1}
Theorem~\ref{main-rate-as-reg-VAMe-1} holds if we replace in the hypothesis $\honelambda$ with  $\honestarlambda$   and  
in the rates $\sigmaonelambda$ with $\sigmaonestarlambda$. In the proof we apply \eqref{main-ineq-VAMe-lambdastarone} 
instead of \eqref{main-ineq-VAMe-lambdaone}.
\end{remark}

\begin{remark}\label{main-theorem-remark-2}
By Lemma~\ref{htwolambda+hthreelambda-implies-honelambda}, Theorem~\ref{main-rate-as-reg-VAMe-1} also holds if we 
assume $\htwolambda$ and $\hthreelambda$ instead of $\honelambda$. Then $\sigmaonelambda$ is given by \eqref{def-sigmaonelambda}.
\end{remark}

The second main result shows that, given a rate of asymptotic regularity of $(x_n)$, one can compute, 
under some quantitative hypotheses on the parameter sequences, rates 
of $(J_{\lambda_n}^A)$-asymptotic regularity and of 
$J_{\lambda_m}^A$-asymptotic regularity for every $m\in\N$.

\begin{theorem}\label{main-rate-as-reg-VAMe-2}
Suppose that $\Phi$ is a rate of asymptotic regularity of $(x_n)$,  $\htwoen$ holds,  $z\in \mathrm{zer}A$,  and 
$K_z\in \N^*$  satisfies \eqref{def-Kz-main}.
\begin{enumerate}
\item\label{main-rate-as-reg-VAMe-2-i} Assume that $\hthreealpha$ holds. Define $\Psi:\N\to \N$ by 
\begin{align*}
\Psi(k)=\max\{\sigmathreealpha(6K_z(k+1)-1),\Phi(3k+2),\sigmatwoen(3k+2)\}.
\end{align*}
Then $\Psi$ is a rate of $\left(J_{\lambda_n}^A\right)$-asymptotic regularity of $(x_n)$.
\item\label{main-rate-as-reg-VAMe-2-ii}  Assume that $\hthreealpha$ and $\htwolambda$ both hold. Define, for every $m\in\N$, $\Theta_m:\N\to \N$  by
\begin{align*}
\Theta_m(k)=\max\{N_\Lambda,\Psi(\boundlambdam\Lambda(k+1)-1),\Psi(2k+1)\},
\end{align*}
where $\boundlambdam\in\N^*$ is such that $\boundlambdam\geq \lambda_m$.

Then, for every $m\in\N$, 
$\Theta_m$ is a rate of $J_{\lambda_m}^A$-asymptotic regularity of $(x_n)$.
\end{enumerate}
\end{theorem}
\begin{proof}
\begin{enumerate}
\item Remark first that for all $n\in\N$, 
\begin{align*}
\|J^A_{\lambda_n}x_n - x_{n+1}\| &=  \|J^A_{\lambda_n}x_n-(\alpha_n f(x_n)+(1-\alpha_n)J^A_{\lambda_n}x_n+e_n)\| \\
&  = 
\|\alpha_n(J^A_{\lambda_n}x_n-f(x_n))-e_n\|\le \alpha_n\|J^A_{\lambda_n}x_n-f(x_n)\| +\|e_n\| \\
& \le   2\alpha_n K_z +\|e_n\|\quad 
\text{by Lemmas~\ref{xn-bound-as-reg}.\eqref{Jlambdamxnxnfxnz-bound-2Kzn} and \ref{honeen-htwoen-htwoen-Kz}.\eqref{hthreeen-Kz}}.
\end{align*}
It follows that for all $n \ge \Psi(k)$, 
\begin{align*}
\|J^A_{\lambda_n}x_n - x_n\| &\le  \|J^A_{\lambda_n}x_n - x_{n+1}\|+\|x_{n+1}-x_n\|\le 2\alpha_n K_z 
+ \|x_{n+1}-x_n\| +\|e_n\|\\
&\le \frac{1}{3(k+1)}+\frac{1}{3(k+1)}+\frac{1}{3(k+1)}=\frac{1}{k+1},
\end{align*}
by $\hthreealpha$, the fact that $\Phi$ is a rate of asymptotic regularity of $(x_n)$, and 
$\htwoen$.

Thus, $\Psi$ is a rate of  $\left(J_{\lambda_n}^A\right)$-asymptotic regularity of $(x_n)$.
\item  Let $m\in\N$. 
For all $n\in\N$, we have that 
\begin{align*}
\|J^A_{\lambda_m}x_n - x_n\| 
&\le \|J^A_{\lambda_m}x_n - J^A_{\lambda_n}x_n\|+\|J^A_{\lambda_n}x_n-x_n\| \\
&\le \frac{\left|\lambda_n-\lambda_m\right|}{\lambda_n} \|x_n-J^A_{\lambda_n} x_n\|+\|J^A_{\lambda_n}x_n-x_n\| \quad \text{by \eqref{J-ineq}}\\
&= \left(\frac{\left|\lambda_n-\lambda_m\right|}{\lambda_n}+1\right)\|J^A_{\lambda_n} x_n-x_n\|. 
\end{align*}

Let $n\ge \Theta_m(k)$. We have two cases:
\begin{enumerate}
\item $\lambda_m \ge \lambda_n$. Then $\frac{\left|\lambda_n-\lambda_m\right|}{\lambda_n}+1=\frac{\lambda_m-\lambda_n}{\lambda_n}+1
=\frac{\lambda_m}{\lambda_n}$, so 
\begin{align*}
\|J^A_{\lambda_m}x_n - x_n\| &\le  \frac{\lambda_m}{\lambda_n}\|J^A_{\lambda_n}x_n-x_n\| 
\stackrel{\htwolambda}\le  \boundlambdam\Lambda\|J^A_{\lambda_n}x_n-x_n\| \le \frac{1}{k+1},
\end{align*}
as $n\geq \Psi(\boundlambdam\Lambda(k+1)-1)$.
\item $\lambda_m<\lambda_n$. Then $\frac{\left|\lambda_n-\lambda_m\right|}{\lambda_n}+1=\frac{\lambda_n-\lambda_m}{\lambda_n}+1
=2-\frac{\lambda_m}{\lambda_n}$, so 
\begin{align*}
\|J^A_{\lambda_m}x_n - x_n\| &\le   \left(2-\frac{\lambda_m}{\lambda_n}\right)\|J^A_{\lambda_n}x_n-x_n\| 
< 2\|J^A_{\lambda_n}x_n-x_n\| \le \frac{1}{k+1},
\end{align*}
as $n\geq \Psi(2k+1)$.
\end{enumerate}
Thus,  $\Theta_m$ is a rate of $J_{\lambda_m}^A$-asymptotic regularity of $(x_n)$.
\end{enumerate}
\end{proof}

As it is the case with applications of proof mining, we obtain effective uniform rates that 
have a very weak dependency on the normed space $X$ and the $m$-accretive operator $A$, only via 
$K_z$ given by \eqref{def-Kz-main} for some zero $z$ of $A$. The rates are computed for arbitrary 
parameter sequences $(\alpha_n)$, $(\lambda_n)$, $(e_n)$ satisfying the 
quantitative hypotheses  stated in Theorems~\ref{main-rate-as-reg-VAMe-1}, \ref{main-rate-as-reg-VAMe-2} 
or Remarks~\ref{main-theorem-remark-1}, \ref{main-theorem-remark-2}
 and  depend on the different moduli associated to these hypotheses. 
 As we shall see in Subsection~\ref{linear-rates-as-reg}, we get linear rates for concrete 
 instances of such sequences. 

Furthermore, if one forgets about the quantitative aspects, 
one gets, as an immediate consequence, qualitative asymptotic regularity results for the VAMe 
iteration $(x_n)$.

\begin{corollary}\label{cor-as-reg-qualitative-1}
Assume that $\sum\limits_{n=0}^{\infty} \alpha_n=\infty$,  
$\sum\limits_{n=0}^{\infty} |\alpha_n-\alpha_{n+1}|<\infty$, $\sum\limits_{n=0}^{\infty} \|e_n\|<\infty$,  and one of the following holds: 
\begin{center}
(a)  $\sum\limits_{n=0}^{\infty} \left|1-\frac{\lambda_{n+1}}{\lambda_n}\right|<\infty$, 
\,\, (b) $\sum\limits_{n=0}^{\infty} \left|1-\frac{\lambda_n}{\lambda_{n+1}}\right|<\infty$, 
\,\,  (c) $\inf\limits_{n\in\N}\lambda_n >0$ and $\sum\limits_{n=0}^{\infty} |\lambda_n-\lambda_{n+1}|<\infty$.
\end{center}
Then $\lim\limits_{n\to\infty} \|x_n-x_{n+1}\|=0$. 
\end{corollary}

\begin{corollary}\label{cor-as-reg-qualitative-2}
Suppose that $\lim\limits_{n\to\infty} \|x_n-x_{n+1}\|=0$ and  $\lim\limits_{n\to\infty}\|e_n\|=0$. 
\begin{enumerate}
\item If $\lim\limits_{n\to\infty}\alpha_n=0$, then $\lim\limits_{n\to\infty} \|x_n-J_{\lambda_n}^Ax_n\|=0$. 
\item If $\lim\limits_{n\to\infty}\alpha_n=0$ and $\inf\limits_{n\in\N}\lambda_n >0$ hold, then 
$\lim\limits_{n\to\infty} \|x_n-J_{\lambda_m}^Ax_n\|=0$ for every $m\in\N$.
\end{enumerate}
\end{corollary}

\subsection{Rates for the VAM iteration}\label{rates-VAM-iteration}

By letting $e_n=0$ for all $n\in\N$, the VAMe iteration becomes the VAM iteration:
$$
\mathrm{VAM} \qquad x_0=x\in X, \quad x_{n+1}=\alpha_n f(x_n) +(1-\alpha_n)J_{\lambda_n}^Ax_n,
$$
where $A$ is an $m$-accretive operator.

Let $z\in \mathrm{zer}A$ and $\KzVAM\in \N^*$ satisfy
\begin{equation}\label{def-KzVAM}
\KzVAM\geq \max\left\{\|x-z\|,\frac{\|f(z)-z\|}{1-\alpha}\right\}.
\end{equation}

By a slight modification of the proofs of Theorems~\ref{main-rate-as-reg-VAMe-1}, \ref{main-rate-as-reg-VAMe-2}, 
taking into account that $e_n=0$ for all $n\in\N$ and that Lemma~\ref{honeen-htwoen-htwoen-Kz}.\eqref{hthreeen-Kz} 
holds with $E=0$, we obtain rates for the VAM iteration.

\begin{proposition}\label{rate-as-reg-VAM-1}
Assume that $\honealpha$, $\htwoalpha$, $\honelambda$ hold and define
\begin{align*}
\chiVAM(k) & =\max\{\sigmatwoalpha(4\KzVAM(k+1)-1),\sigmaonelambda(4\KzVAM(k+1)-1)\}, \\[1mm]
\phiVAM(k)  & = \sigmaonealpha\left(\left \lceil 
\frac{\chiVAM(2k+1)+1+\lceil \ln(4\KzVAM(k+1))\rceil}{1-\alpha} \right\rceil +1\right).
\end{align*}
Then $\phiVAM$ is a rate of asymptotic regularity of the VAM iteration $(x_n)$.
\end{proposition}

Remarks~\ref{main-theorem-remark-1}, \ref{main-theorem-remark-2} are true for the VAM iteration too.

\begin{proposition}\label{rate-as-reg-VAM-2}
Let $\phiVAM$ be a rate of asymptotic regularity of $(x_n)$. 
Define
\begin{align*}
\psiVAM(k) &= \max\{\sigmathreealpha(4\KzVAM(k+1)-1),\phiVAM(2k+1)\} 
\quad \text{if~}\hthreealpha \text{~holds,}
\end{align*}
and, for $m\in\N$ and $\boundlambdam\in\N^*$ such that $\boundlambdam\geq \lambda_m$,
\begin{align*}
\thetaVAM_m(k) & =\max\{N_\Lambda,\psiVAM(\boundlambdam\Lambda(k+1)-1),\psiVAM(2k+1)\} 
\quad \text{if both ~}\hthreealpha \text{~and~} \htwolambda \text{~hold.}
\end{align*}
Then $\psiVAM$ is a rate of $\left(J_{\lambda_n}^A\right)$-asymptotic regularity of $(x_n)$ 
and, for every $m\in\N$, $\thetaVAM_m$ is a rate of $J_{\lambda_m}^A$-asymptotic regularity 
of $(x_n)$. 
\end{proposition}

Obviously, Corollaries~\ref{cor-as-reg-qualitative-1}, 
\ref{cor-as-reg-qualitative-2} (with the hypotheses $\sum\limits_{n=0}^{\infty} \|e_n\|<\infty$, $\lim\limits_{n\to\infty}\|e_n\|=0$ 
removed) hold also for the VAM iteration. 

We remark that in \cite{XuAltImtSou22} the VAM iteration $(x_n)$ is studied in a more general 
setting, by considering  an accretive operator $A$, an $\alpha$-contraction $f:C\to C$, and $x\in C$, where 
$\emptyset \ne C \subseteq X$ is a nonempty closed convex subset of $X$ satisfying 
$\overline{\mathrm{dom}A} \subseteq C\subseteq  \mathrm{ran}(\mathrm{Id}+\gamma A)$ for all $\gamma>0$. It is easy to 
see that the results from Section~\ref{VAMe-results} specialized to $e_n=0$ hold in this setting
with basically the same proofs. Hence, Propositions~\ref{rate-as-reg-VAM-1}, \ref{rate-as-reg-VAM-2}  are true
in this more general setting, too.

\subsection{Rates for the HPPA iteration}

Another particular case of the VAMe iteration is the (inexact) HPPA iteration:

$$
\mathrm{HPPA} \qquad x_0=x\in X, \quad x_{n+1}=\alpha_n u +(1-\alpha_n)J_{\lambda_n}^Ax_n +e_n,
$$
obtained by letting $f(x)=u\in X$ in the definition \eqref{def-VAMe} of VAMe. Obviously, the 
constant mapping $f(x)=u$ is an $\alpha$-contraction with $\alpha=0$. 

Theorems~\ref{main-rate-as-reg-VAMe-1}, \ref{main-rate-as-reg-VAMe-2} hold for the HPPA iteration with $K_z\in \N^*$ such that 
$$
K_z\geq \max\left\{\|x-z\|, \|u-z\|\right\} + \left\lceil\sum\limits_{i=0}^{\sigmaoneen(0)}\|e_i\|\right\rceil + 1.
$$
Furthermore, Corollaries~\ref{cor-as-reg-qualitative-1}, \ref{cor-as-reg-qualitative-2} are 
true for the HPPA $(x_n)$, too.

By letting $e_n=0$, we get that Propositions~\ref{rate-as-reg-VAM-1}, \ref{rate-as-reg-VAM-2} hold with 
$\KzVAM \in \N^*$ such that 
$$\KzVAM \geq \max\left\{\|x-z\|, \|u-z\|\right\}.$$

Methods of proof mining were applied in \cite{LeuPin21,Pin21}  to the HPPA  iteration associated 
to a maximal monotone operator $A$ in a Hilbert space $X$ to obtain quantitative results on its asymptotic behaviour, including rates of 
$\left(\left(J_{\lambda_n}^A\right), \, J_{\lambda_m}^A (m\in\N)\right)$-asymptotic regularity.

In this paper we compute such rates for the more general setting of $m$-accretive operators in Banach spaces.

\subsection{Linear rates for concrete instances of the parameter sequences}\label{linear-rates-as-reg}

\cite[Lemma~3]{SabSht17} or its slight variation, Lemma~\ref{lem:sabach-shtern-v2}, were applied recently  to obtain linear rates of asymptotic regularity for the Tikhonov-Mann and 
modified Halpern iterations \cite{CheKohLeu23}, the alternating Halpern-Mann iteration \cite{LeuPin24} and for different Halpern-type 
iterations \cite{CheLeu24}. In the sequel, we use  Lemma~\ref{lem:sabach-shtern-v2} to compute
linear rates for the VAMe iteration for two specific choices of the parameter sequences.

In the following, for all $n\in\N$,  
\begin{align*}
\alpha_n=\frac{2}{(1-\alpha)(n+J)},  \qquad \text{where~}J=2\left\lceil\frac{1}{1 - \alpha}\right\rceil.
\end{align*}
As  $(\alpha_n)$ is decreasing, we have that $\alpha_n\leq \alpha_0 = \frac{2}{(1-\alpha)J}\leq 1$. Thus, $\alpha_n$ is  a sequence in 
$[0,1]$. 

\subsubsection{A first example}

For all $n\in\N$, consider 
\begin{align*}
\lambda_n=\lambda >0  \text{~and~} e_n=0.
\end{align*}
Then  $(x_n)$ is the VAM iteration with a single mapping 
$J_\lambda^A$, which is nonexpansive.
It follows that $(x_n)$ is a particular case of the viscosity 
version of the Halpern iteration (where one considers an arbitrary nonexpansive mapping $T$ 
instead of $J_\lambda$) introduced by Xu \cite{Xu04} and studied by Sabach and 
Shtern \cite{SabSht17} under the name of sequential averaging method (SAM). As an application of \cite[Lemma~3]{SabSht17}, 
Sabach and Shtern obtained  linear rates of  ($T$-)asymptotic regularity for SAM. 
Cheval and the second author \cite{CheLeu24} applied Lemma~\ref{lem:sabach-shtern-v2}  to compute 
such linear rates in the more general setting of $W$-hyperbolic spaces; these rates hold in our setting, too.

Consider the following mappings, defined in \cite[Section 3.2, (15), (16)]{CheLeu24}, with notations adapted to this paper:
\begin{align*}
\Phi_0(k) & = 4\KzVAM \left\lceil\frac{1}{1 - \alpha}\right\rceil^2(k+1)-2\left\lceil\frac{1}{1 - \alpha}\right\rceil,\\
\Psi_0(k) & = \left(4\KzVAM\left\lceil\frac{1}{1 - \alpha}\right\rceil^2+
4\KzVAM\left\lceil\frac{1}{1 - \alpha}\right\rceil\right)(k+1)-2\left\lceil\frac{1}{1 - \alpha}\right\rceil,
\end{align*}
where  $z\in \mathrm{zer}A$ and $\KzVAM\in \N^*$ satisfies \eqref{def-KzVAM}.

Then $(x_n)$ is asymptotically regular with rate $\Phi_0$ and $J_{\lambda}^A$-asymptotically regular with rate $\Psi_0$. 
As $\lambda_n=\lambda$ for all $n\in\N$, obviously $(J_{\lambda_n}^A)$-asymptotic regularity and  $J_{\lambda_m}^A$-asymptotic regularity (for $m\in\N$) 
coincide with $J_{\lambda}^A$-asymptotic regularity of $(x_n)$.

\subsubsection{A second example}

Let us take, for all $n\in\N$,
\begin{align*}
\lambda_n & =\frac{n+J}{n+J-1} \text{~and~} e_n=\frac{1}{(n+J)^2}\ebase,  \qquad \text{where~}\ebase\in X.
\end{align*}

Since $\displaystyle \sum\limits_{n=0}^\infty \frac1{(n+J)^2} < \frac1{J-1}$, it follows that $\hthreeen $ holds with 
$\displaystyle \sigmathreeen=\left\lceil\frac{\|\ebase\|}{J-1}\right\rceil$. 

Let $z\in \text{zer}A$ and $K_z\in \N^*$ satisfying 
\begin{equation}
K_z\geq \max\left\{\|x-z\|,\frac{\|f(z)-z\|}{1-\alpha}\right\} + \left\lceil\frac{\|\ebase\|}{J-1}\right\rceil.
\end{equation}

\begin{proposition}\label{linear-rate-as-reg-VAM-example}
For all $n\in\N$, 
\begin{align}\label{linear-ineq-VAM-example}
\|x_{n+1}-x_n\|\le \frac{3JK_z+\|\ebase\|}{(1-\alpha)(n + J)}.
\end{align}
Thus, 
\begin{align*}
\Phi_0(k)& =\left(3JK_z+\left\lceil\|\ebase\|\right\rceil\right)\left\lceil\frac{1}{1-\alpha}\right\rceil(k+1) -J \\
& = 6K_z \left\lceil\frac{1}{1-\alpha}\right\rceil^2(k+1) + \left\lceil\|\ebase\|\right\rceil\left\lceil\frac{1}{1-\alpha}\right\rceil(k+1)- 2\left\lceil\frac{1}{1 - \alpha}\right\rceil
\end{align*}
is a linear rate of asymptotic regularity of $(x_n)$.
\end{proposition}
\begin{proof}
By Lemma~\ref{honeen-htwoen-htwoen-Kz}.\eqref{hthreeen-Kz}, we have that Lemma~\ref{xn-bound-as-reg} 
and \eqref{main-ineq-VAMe-lambdaone} hold with $K_z$ defined as above instead of $K_{z,n}$ or $K_{z,n+1}$. 
Applying \eqref{main-ineq-VAMe-lambdaone}, we get that for all $n\in\N$,
\begin{align*}
\|x_{n+2}-x_{n+1}\| & \le  (1-(1-\alpha)\alpha_{n+1})\|x_{n+1}-x_n\| + P_z,
\end{align*}
where 
\begin{align*}
P_z &= 2K_z\left(|\alpha_{n+1}-\alpha_n|
+ (1-\alpha_{n+1})\left| 1-\frac{\lambda_{n+1}}{\lambda_n} \right|\right) + \|e_{n+1}-e_n\|\\
& = |\alpha_{n+1}-\alpha_n|\left(2K_z
\left(1+\frac{(1-\alpha_{n+1})}{|\alpha_{n+1}-\alpha_n|}\left| 1-\frac{\lambda_{n+1}}{\lambda_n} \right| \right)+ 
\frac{\|e_{n+1}-e_n\|}{|\alpha_{n+1}-\alpha_n|}\right).
\end{align*}
As 
\begin{align*}
|\alpha_{n+1}-\alpha_n| & = \alpha_n-\alpha_{n+1}=\frac{2}{(1-\alpha)(n+J)(n+1+J)}, \\
1-\alpha_{n+1} & = \frac{(1-\alpha)(n+1+J)-2}{(1-\alpha)(n+1+J)},\\
\left| 1-\frac{\lambda_{n+1}}{\lambda_n} \right| &= 1-\frac{\lambda_{n+1}}{\lambda_n}=\frac{1}{(n+J)^2}, \\
\end{align*}
we have that
\begin{align*}
\frac{(1-\alpha_{n+1})}{|\alpha_{n+1}-\alpha_n|}\left| 1-\frac{\lambda_{n+1}}{\lambda_n} \right| & = 
\frac{\big((1-\alpha)(n+1+J)-2\big)(n+J)}2 \cdot \frac{1}{(n+J)^2}\\
& \leq \frac{n+J-1}{2(n+J)}<\frac12.
\end{align*}
Furthermore, 
\begin{align*}
\frac{\|e_{n+1}-e_n\|}{|\alpha_{n+1}-\alpha_n|} & = \frac{(2(n+J)+1)\|\ebase\|(1-\alpha)}{2(n+J)(n+J+1)} 
\leq \frac{\|\ebase\|}{n+J}\leq \frac{\|\ebase\|}J.
\end{align*}
It follows that for all $n\in\N$,
\begin{align*}
\|x_{n+2}-x_{n+1}\| & <  (1-(1-\alpha)\alpha_{n+1})\|x_{n+1}-x_n\| +  (\alpha_n-\alpha_{n+1})\left(3K_z + \frac{\|\ebase\|}J\right). 
\end{align*}
One can easily see that Lemma~\ref{lem:sabach-shtern-v2} can be applied  with 
\begin{center}
$s_n= \|x_{n+1}-x_n\|$,  $L=3K_z+\frac{\|\ebase\|}J$,  $N=2$,  $J=2\left\lceil\frac{1}{1 - \alpha}\right\rceil$,  
$\gamma=1-\alpha$,  $a_n=\alpha_n$, 
$c_n=3K_z+\frac{\|\ebase\|}J$ 
\end{center}
to conclude that \eqref{linear-ineq-VAM-example} holds and, 
as a consequence, $\Phi_0$ is a rate of asymptotic regularity of $(x_n)$.
\end{proof}

\begin{proposition} \label{linear-rate-J-as-reg-VAM-example}
Define 
\begin{align*}
\Psi_0(k)  & =  18K_z \left\lceil\frac{1}{1-\alpha}\right\rceil^2(k+1) + 
3\left\lceil\|\ebase\|\right\rceil\left\lceil\frac{1}{1-\alpha}\right\rceil(k+1)- 
2\left\lceil\frac{1}{1 - \alpha}\right\rceil,\\
\Theta_0(k) & = 
36K_z \left\lceil\frac{1}{1-\alpha}\right\rceil^2(k+1) + 6\left\lceil\|\ebase\|\right\rceil\left\lceil\frac{1}{1-\alpha}\right\rceil(k+1) - 
2\left\lceil\frac{1}{1 - \alpha}\right\rceil.
\end{align*}
Then $\Psi_0$ is a linear rate of $\left(J_{\lambda_n}^A\right)$-asymptotic regularity of $(x_n)$ and 
$\Theta_0$ is a linear rate of $J_{\lambda_m}^A$-asymptotic regularity of $(x_n)$ for every $m\in\N$. 
\end{proposition}
\begin{proof}
We can apply Theorem~\ref{main-rate-as-reg-VAMe-2}, as  $\hthreealpha$ holds 
with $\sigmathreealpha(k)=Jk$, $\htwoen$ holds with 
$\sigmatwoen(k)=\max\left\{\left\lceil\sqrt{\|\ebase\|(k+1)}\right\rceil-J, 0 \right\}$, 
$\htwolambda$ holds with $\Lambda=1$, $N_\Lambda=0$, and $\boundlambdam=2 \geq \lambda_m$ for all $m\in\N$. 

Using also Theorem~\ref{linear-rate-as-reg-VAM-example}, it follows that 
$(x_n)$ is $\left(J_{\lambda_n}^A\right)$-asymptotically 
regular with rate 
\begin{align*}
\Psi(k) &= \max\{\sigmathreealpha(6K_z(k+1)-1), \Phi_0(3k+2), \sigmatwoen(3k+2)\}\}
\end{align*}

Since 
\begin{align*}
\sigmathreealpha(6K_z(k+1)-1) & = 6JK_z(k+1)-J,\\
\Phi_0(3k+2) & = 9JK_z\left\lceil\frac{1}{1-\alpha}\right\rceil(k+1) + 3\left\lceil\|\ebase\|\right\rceil\left\lceil\frac{1}{1-\alpha}\right\rceil(k+1)-J,\\
\sigmatwoen(3k+2) & = \max\left\{\left\lceil\sqrt{3\|\ebase\|(k+1)}\right\rceil-J, 0 \right\},
\end{align*}
we have that $\sigmathreealpha(6K_z(k+1)-1), \sigmatwoen(3k+2) < \Phi_0(3k+2)$, hence 
$$\Psi(k) = \Phi_0(3k+2) = \Psi_0(k).$$ 
Applying Proposition~\ref{main-rate-as-reg-VAMe-2}.\eqref{main-rate-as-reg-VAMe-2-ii}, we get that for every $m\in\N$, 
$(x_n)$ is $J_{\lambda_m}^A$-asymptotically 
regular with rate 
\begin{align*}
\Theta_m(k) & =\max\{N_\Lambda,\Psi_0(\boundlambdam\Lambda(k+1)-1),\Psi_0(2k+1)\}=\Psi_0(2k+1) = \Theta_0(k).
\end{align*}
\end{proof}

\mbox{}

\section*{Acknowledgements}

Paulo Firmino acknowledges the support of FCT – Fundação para a Ciência e Tecnologia through a doctoral scholarship with reference number 2022.12585.BD as well as the support of Fundação para a Ciência e Tecnologia via the research center CMAFcIO (Universidade de Lisboa) under funding \url{https://doi.org/10.54499/UIDB/04561/2020}.

\end{document}